\documentclass[12pt]{article}

\usepackage{amsmath}
\usepackage{latexsym}
\usepackage{amsfonts}
\usepackage{amssymb}
\usepackage{amscd}
\usepackage{theorem}
\usepackage{a4wide}

\newtheorem{theorem}{Theorem}[section]
\newtheorem{lemma}[theorem]{Lemma}
\newtheorem{proposition}[theorem]{Proposition}
\newtheorem{corollary}[theorem]{Corollary}

{\theorembodyfont{\rmfamily}
  
  \newtheorem{example}[theorem]{Example}

  \newtheorem{remark}[theorem]{Remark}
}

\newenvironment{proof}{    
  \noindent
  \textbf{Proof.}}{
  \hfill $\Box$
  \vspace{3mm}
}

\numberwithin{equation}{section}

\newcommand{\N}{\mathbb{N}} 
\newcommand{\R}{\mathbb{R}} 
\newcommand{\C}{\mathbb{C}} 

 \DeclareMathOperator{\re}{Re} 
 \DeclareMathOperator{\Cot}{Cot} \DeclareMathOperator{\Id}{Id}

\begin{document}

\title{Abscissas of weak convergence of vector valued Dirichlet series}

\author{Jos\'{e} Bonet}

\date{}

\maketitle

\begin{center}
\textit{Dedicated to my friend Prof.\ Manuel Maestre on the occasion of his 60th birthday}
\end{center}

\begin{abstract}
The abscissas of convergence, uniform convergence and absolute convergence of vector valued Dirichlet series with respect to the original topology and with respect to the weak topology $\sigma(X,X')$ of a locally convex space $X$, in particular of a  Banach space $X$, are compared. The relation of their coincidence with  geometric or topological properties of the underlying space $X$ is investigated. Cotype in the context of Banach spaces, and nuclearity and certain topological invariants for Fr\'echet spaces play a relevant role.
\end{abstract}

\renewcommand{\thefootnote}{}


\section{Introduction and preliminaries}

The general theory of Dirichlet series was developed at the beginning of the last century by Bohr, Hardy, Landau and Riesz, among others. Recently the field showed remarkable advances, in particular
combining functional analytical and complex analytical tools. We refer to the book \cite{Queffelec_book}, the articles \cite{Boas_Football}, \cite{Hedenmalm} and \cite{Queffelec},
 and the references therein for more information. The research on vector valued Dirichlet series with coefficients in a Banach space was initiated
 by Defant, Garc\'{\i}a, Maestre and P\'erez-Garc\'{\i}a in \cite{defantGMP}, in which the width of the largest possible strip on which a Dirichlet series with coefficients in a Banach space converges uniformly but not absolutely is investigated. See also the survey paper \cite{DefantGMS} and the references in the recent paper \cite{DefantSS}. Our purpose here is to compare
 the abscissas of convergence of vector valued Dirichlet series for the original topology and for the weak topology and to relate their behaviour with the geometry of the underlying space. With this aim in mind, locally convex spaces seem to be the proper context.

 We prove that the abscissas of convergence and of uniform convergence of a Dirichlet series $D = \sum_n a_n \dfrac{1}{n^s}$, with coefficients $a_n \in X$, for the original topology and for the weak topology on a sequentially complete locally convex space $X$ coincide (Proposition \ref{c_equals_weak_c} and Corollary \ref{uequalsweaku}). In a Banach space $X$, if the abscissa of convergence of the scalar series $\sum_n x'(a_n) \dfrac{1}{n^s}$ is finite for every $x' \in X'$, then the abscissa of convergence of $D$ is also finite, as we show in Corollary \ref{Banach}. This is not the case for non-normable Fr\'echet spaces. Those Fr\'echet spaces which share this behaviour are characterized in Theorem \ref{lbinfty} in terms of a topological invariant of Vogt \cite{Vogt} of $(DN)$-type. The abscissas of absolute convergence $\sigma_a(D)$ for the original topology and $\sigma^w_a(D)$ for the weak topology of $D$ do not coincide in general. We introduce the \textit{gap for absolute convergence of Dirichlet series in $X$} as $G_a(X):= \sup_{D} (\sigma_a(D) -  \sigma^w_a(D))$, the supremum  taken over all the Dirichlet series $D$ with coefficients in $X$ such that the abscissa of convergence is finite. We show in Proposition \ref{gaplowerboundBanach} that $G_a(X) \geq 1/2$ for every infinite dimensional Banach space $X$, and we determine $G_a(X)$ for infinite dimensional Banach spaces $X$ is terms of the cotype of $X$ in Theorems \ref{cotypeBanach} and \ref{nofinitecotype}. These two results should be compared with \cite[Theorem 1]{defantGMP}. A Fr\'echet space $X$ is nuclear if and only if $G_a(X)=0$ if and only if $G_a(X) < 1/2$, by Theorem \ref{nuclear}.

In what follows $X$ denotes a sequentially complete locally convex space, that will always be assumed to be Hausdorff. The system of all continuous seminorms $\alpha: X \rightarrow [0,\infty[$ defining the topology of $X$ will be denoted by $cs(X)$, and $X'$ stands for the topological dual of $X$. We write $\sigma(X,X')$
for the weak topology in $X$. All the topologies of the dual pair $(X,X')$ have the same bounded sets by Mackey's theorem \cite[Theorem 23.15]{Meise_Vogt}. If $\Omega$ is an open subset of $\C$, the space of holomorphic functions on $\Omega$ with values in $X$ will be denoted by $H(\Omega,X)$; see \cite{Jarchow}. Moreover, $H_\infty(\Omega,X)$ stands for the space of bounded holomorphic functions.
Our notation for locally convex spaces, Banach spaces and functional analysis is standard.  See \cite{albiackalton} \cite{Defant_Floret}, \cite{Diestel}, \cite{DiJaTo}, \cite{Floret_Wloka}, \cite{Jarchow}, \cite{Meise_Vogt}.

A Dirichlet series in a sequentially complete locally convex space $X$ is a series of the form $D = \sum_n a_n \dfrac{1}{n^s}$ with
coefficients $a_n \in X$ and variable $s \in  \mathbb{C}$. The abscissas of convergence, uniform convergence and absolute convergence of $D$ are defined as follows:

$$
\sigma_c(D):= \inf \{ r \ \big| \ \sum_n a_n \dfrac{1}{n^s} \ {\rm converges \ in } \ X \ {\rm on } \ [\re > r] \},
$$
$$
\sigma_u(D):= \inf \{ r \ \big| \ \sum_n a_n \dfrac{1}{n^s} \ {\rm converges \ uniformly  \ in } \ X \ {\rm on } \ [\re > r] \},
$$
$$
\sigma_a(D):= \inf \{ r \ \big| \ \sum_n a_n \dfrac{1}{n^s} \ {\rm converges \ absolutely \ in } \ X \ {\rm on } \ [\re > r] \}.
$$
Here the infima are taken in the extended real line. When the Dirichlet series is nowhere convergent, the three abscissas are $+\infty$.

Given $x' \in X'$, one can consider for the scalar Dirichlet series $x'(D) = \sum_n x'(a_n) \dfrac{1}{n^s}$ the three abscissas of convergence $\sigma_i(x'(D)), i=c,u,a$. It is clear that
$\sigma_i(x'(D)) \leq \sigma_i(D), i=c,u,a$; therefore $\sigma^w_i(D):= \sup_{x' \in X'} \sigma_i(x'(D))$ is dominated by $\sigma_i(D)$ for each vector valued Dirichlet series $D$ in $X$. Moreover, it is  easy to see that each $\sigma^w_i(D)$ coincides with the corresponding abscissa of convergence of the Dirichlet series $D$ when the convergence is considered in the weak topology $\sigma(X,X')$. We call $\sigma^w_i(D)$ the \textit{abscissas of weak convergence} of the Dirichlet series $D$. If $X$ and $Y$ are isomorphic, then the abscissas of (weak) convergence of $X$ and $Y$ coincide. We compare the behaviour of $\sigma^w_i(D)$ and $\sigma_i(D)$ for $i=c,u,a$ for all Dirichlet vector valued series in $X$, and relate this behaviour with the topological structure of the space $X$.

The following vector valued Abel identity is needed below.

\begin{lemma}\label{abelvv}
Let $(b_n)_n$ be a sequence in a sequentially complete locally convex space $X$. For $x \in [1, \infty[$, set $A(x):= \sum_{n\leq x} b_n$. Let $\varphi:[1,x] \rightarrow \C$ be a $C^1$-function. Then
$$
\sum_{n \leq x} b_n \varphi(n) =A(x) \varphi(x) - \int_{1}^{x} A(t) \varphi'(t) dt.
$$
\end{lemma}
\begin{proof}
All the elements in the equality are well defined in $X$, except the integral. However, if $m \leq x < m+1, m \in \N$, then $A(t) \varphi'(t) = (\sum_{n=1}^{j} b_n) \varphi'(t)$ for $t \in [j,j+1[, j=1,...,m-1,$ and $A(t) \varphi'(t) = (\sum_{n=1}^{m} b_n) \varphi'(t)$  for $t \in [m,x], j=m$. This implies that the integral exists in $X$. The result now follows from the scalar valued case, see e.g.\ \cite[Proposition 1.3.6]{Jameson}, after evaluating both sides on each $x' \in  X'$.
\end{proof}

The proof of the next result now follows as in the scalar case \cite{Apostol}, \cite{Jameson}, \cite{Queffelec_book}.

\begin{proposition} \label{elemnetaryprop}
Let $D = \sum_n a_n \dfrac{1}{n^s}$ be a Dirichlet series with coefficients in
a sequentially complete locally convex space $X$ and $s(0) \in \C$.
\begin{itemize}

\item[(i)]  If $\sum_n a_n \dfrac{1}{n^{s(0)}}$ converges, then $\sum_n a_n \dfrac{1}{n^s}$ converges for $s \in \C, \re s > \re s(0)$.

\item[(ii)] If $\sum_n a_n \dfrac{1}{n^{s(0)}}$ converges, then $\sum_n a_n \dfrac{1}{n^s}$ converges absolutely for $s \in \C$ with $ \re s > \re s(0) + 1$.

\item[(iii)] If $\sum_n a_n \dfrac{1}{n^{s(0)}}$ converges absolutely, then $\sum_n a_n \dfrac{1}{n^s}$ converges absolutely for $s \in \C$ with $ \re s \geq \re s(0)$.

\item[(iv)] $-\infty \leq \sigma_c(D) \leq \sigma_u(D) \leq \sigma_a(D) \leq +\infty$.

\item[(v)] $\sigma_a(D) \leq \sigma_c(D) + 1$.

\end{itemize}

\end{proposition}

\begin{proposition} \label{holomorphic}
If the Dirichlet series $D(s):= \sum_n a_n \dfrac{1}{n^s}, \ a_n \in X,$ satisfies $\sigma_c(D) > -\infty$, then $D(s)$ defines a holomorphic function on $[\re s > \sigma_c(D)]$; that is $D \in H([\re s > \sigma_c(D)],X)$. Moreover $D'(s)=-\sum_n (a_n \log n)/ n^s), \ \re s > \sigma_c(D)$.
\end{proposition}
\begin{proof}
By \cite[Proposition 1.7.10]{Jameson} or \cite[Jensen's Lemma 4.1.1]{Queffelec_book}, for each $x' \in X'$, the function $x'(D)(s)= \sum_n x'(a_n) \dfrac{1}{n^s}$ is holomorphic on $[\re s > \sigma_c(D)]$. We can apply Grosse-Erdmann \cite[Theorem 1]{Grosse}, that is an extension of a classical result of A.E. Taylor and Grothendieck, to conclude that $D \in H([\re s > \sigma_c(D)],X)$.
\end{proof}

\section{Abscissas of convergence and uniform convergence} \label{sect_c_and_u}

\begin{proposition} \label{c_equals_weak_c}
Let $D = \sum_n a_n \dfrac{1}{n^s}$ be a Dirichlet series in
a sequentially complete locally convex space $X$. If $\sum_n x'(a_n) \dfrac{1}{n^{s(0)}}, s(0) \in \C,$ converges for every $x' \in X'$, then $\sum_n a_n \dfrac{1}{n^s}$ converges in $X$ for each $s \in  \C$ with $\re s > \re s(0)$.

In particular, $\sigma^w_c(D):= \sup_{x' \in X'} \sigma_c(x'(D)) = \sigma_c(D)$.
\end{proposition}
\begin{proof}
Fix $s \in \C, \re s > \re s(0)$ and put $b_n:=a_n/n^{s(0)}, \varphi(x):=x^{-(s-s(0))}, A(x):= \sum_{n \leq x} b_n, x \geq 1,$ and apply Lemma \ref{abelvv} for $N < M$, to get
$$
\sum_{N < n \leq M} a_n \dfrac{1}{n^s} =
A(M)M^{-(s-s(0))} - A(N)N^{-(s-s(0))} + (s-s(0)) \int_N^M A(t) t^{-(s-s(0))-1} dt.
$$
The sequence $(x'(A(N)))_N$ is bounded for each $x' \in X'$, since it is convergent. By Mackey's theorem
\cite[Theorem 23.15]{Meise_Vogt}, $(A(N))_N$ is bounded in $X$ and, for each $\alpha \in cs(X)$ there is $K_\alpha>0$ such that $\alpha(A(N)) \leq K_\alpha$ for each $N \in \N$. On the other hand, for each $N < M$, $\int_N^M t^{-(\re s- \re s(0))-1} dt \leq N^{-(\re s- \re s(0))}/(\re s- \re s(0))$. Therefore
$$
\alpha(\sum_{N < n \leq M} a_n \dfrac{1}{n^s}) \leq 2 K_\alpha N^{-(\re s- \re s(0))} \left(1+ \frac{|s-s(0)|}{\re s- \re s(0)}\right),
$$
which tends to $0$ as $N$ tends to $\infty$.
\end{proof}

The proof of Proposition \ref{c_equals_weak_c} also shows that if $\sum_n a_n \dfrac{1}{n^s}$ does not converge in $X$ for all $s \in \C$ (i.e.\ $\sigma_c(D)=+\infty$), then $\sigma^w_c(D)=+\infty$.
However, it might happen that $\sigma_c(D)=+\infty$ and $\sigma_c(x'(D))< +\infty$ for each $x' \in X'$, although the supremum of these values as $x'$ runs in $X'$ must be infinity. Indeed, take $X= \C^{\N}$ the Fr\'echet space of all the complex sequences $x=(x_i)_i$ endowed with its natural Fr\'echet topology of pointwise convergence, $a_n=(n^{i-1})_i =(1,n,n^2,n^3,...)$ and $D:= \sum_n a_n \dfrac{1}{n^s} = (D_i)_i = (\sum_n \dfrac{n^{i-1}}{n^s})_i$. For $s=i$, we have $D_i(i)=\sum_n 1/n$, hence $\sum_n a_n \dfrac{1}{n^s}$ does not converge in $\omega$ for all $s \in \C$. On the other hand, given $x'=(x'_i)_i \neq 0$ in $X'$, there is $i(x') \in \N$ such that $x'_{i(x')} \neq 0$ and $x'_i = 0$ for $i>i(x')$. Therefore, if $s \in \C, \re s > i(x')$, we get $\sum_n \frac{|x'(a_n)|}{|n^s|} < +\infty$, and $\sigma_c(x'(D)) \leq \sigma_a(x'(D)) < +\infty$. This phenomenon cannot happen for Banach spaces, as we will show now.

\vspace{.1cm}

A Hausdorff locally convex space is said to satisfy the \textit{countable neighbourhood property} (see \cite[8.3.4]{PCB} or \cite[page 478]{Defant_Floret}) if for every sequence $(\alpha_j)_j \subset cs(X)$ there are $\alpha \in cs(X)$ and $(\lambda_j)_j \subset ]0,+\infty[$ such that $\alpha_j(x) \leq \lambda_j \alpha(x)$ for each $x \in X$. Clearly every Banach space $X$ and, more generally, every (DF)-space in the sense of Grothendiek (in particular the strong dual of every Fr\'echet space) satisfies the countable neighbourhood property. A Fr\'echet space has the countable neighbourhood property if and only if it is a Banach space. See \cite[Section 8.3]{PCB} or \cite[Chapter 25]{Meise_Vogt} for more information about (DF)-spaces. The following technical lemma is inspired by \cite[Corollary 5]{Bierstedt_Bonet}.

\begin{lemma}\label{cnp_technical}
Let $X$ be a locally convex space with the countable neighbourhood property. Let $(x_n)_n$ be a sequence in $X$ such that, for each $x' \in X'$ there is $k = k(x') \in \N$ such that $\sup_n n^{-k} |x'(x_n)| < +\infty$. Then there is $k_0 \in \N$ such that $\sup_n n^{-k_0} \alpha(x_n) < +\infty$ for each $\alpha \in cs(X)$.

\end{lemma}
\begin{proof}
We proceed by contradiction and suppose that, for each $j \in \N$ there is $\alpha_j \in cs(X)$ such that $\sup_n n^{-(j+1)} \alpha_j(x_n) = +\infty$. Since $X$ has the countable neighbourhood property, there is $\alpha \in cs(X)$ and there are $\lambda_j \geq 1$ such that $\alpha_j(x) \leq \lambda_j \alpha(x)$ for each $x \in X$. For $j=1$ we select $n(1) \in \N$ such that $\alpha_1(x_{n(1)}) > n(1)^2\lambda_1$.
For $j=2$, select $n(2) \in \N$ such that
$$
n(2)^{-3} \alpha_2(x_{n(2)}) > \lambda_2 + \sum_{s=1}^{n(1)} s \  \alpha_2(x_s).
$$
Observe that our selection implies $n(2) > n(1)$. Proceeding by recurrence, we find a sequence $n(1) < n(2) < ... < n(k) <...$ such that $\alpha_k(x_{n(k)}) > n(k)^{k+1} \lambda_k$ for each $k \in \N$.
Define $v(n):= n(k)^{-k}$ if $n=n(k)$ for some $k$ and $v(n)=0$ otherwise. For each $k \in \N$ and $j \geq k$ we have $v(n(j))/n(j)^{-k} = n(j)^k/n(j)^j \leq 1$. Hence, if we set, $m_k:=\max \{ v(n(s))/n(s)^{-k} \ | \ 1 \leq s \leq k \}$, we conclude that for each $k \in \N$ there is $m_k >0$ such that $v(n) \leq m_k n^{-k}$ for each $n \in\N$.

We show that this implies that the set $C_v := \{ v(n) x_n \ | \ n \in \N \}$ is weakly bounded in $X$. To see this fix $x' \in X'$. By assumption there is $k = k(x') \in \N$ such that $k = k(x') \in \N$ such that $\sup_n n^{-k} |x'(x_n)| < +\infty$. Thus
$$ \sup_n v(n) |x'(x_n)| \leq m_k \sup_n n^{-k} |x'(x_n)| < +\infty.$$
By Mackey's Theorem \cite[Theorem 23.15]{Meise_Vogt}, $C_v$ is bounded, hence there is $M>0$ such that $v(n) \alpha(x_n) \leq M$ for each $n \in \N$. This is a contradiction, since, for each $k \in \N$, we have
$$
n(k) < n(k)^{-k} \lambda_k^{-1} \alpha_k(x_{n(k)}) \leq n(k)^{-k} \alpha(x_{n(k)}) = v(n(k)) \alpha(x_{n(k)}).
$$
\end{proof}

\begin{theorem} \label{cnp}
Let $X$ be a sequentially complete locally convex space with the countable neighbourhood property. If a Dirichlet series $D = \sum_n a_n \dfrac{1}{n^s}$ satisfies $\sigma_c(x'(D)) < +\infty$ for all $x' \in X'$, then $\sigma_c(D) < +\infty$.
\end{theorem}
\begin{proof}
Assume that $D = \sum_n a_n \dfrac{1}{n^s}$ satisfies $\sigma_c(x'(D)) < +\infty$ for all $x' \in X'$.
Then $\sigma_a(x'(D)) \leq \sigma_c(x'(D)) + 1 < +\infty$ for all $x' \in X'$ (see Proposition \ref{elemnetaryprop} (v)). Therefore for each $x' \in X'$ there is $k=k(x')$ such that $\sum_n \frac{|x'(a_n)|}{n^{k}} < +\infty$, hence $\sup_n n^{-k} |x'(a_n)| < +\infty$. We apply Lemma \ref{cnp_technical} to find $k(0) \in \N$ such that $\sup_n n^{-k(0)} \alpha(x_n) < +\infty$ for each $\alpha \in cs(X)$. Thus $\sum_n \frac{\alpha(a_n)}{n^{k(0)+2}} < +\infty$ for each $\alpha \in cs(X)$.
This implies $\sigma_c(D) \leq \sigma_a(D) < +\infty$.
\end{proof}

\begin{corollary} \label{Banach}
If a Dirichlet series $D = \sum_n a_n \dfrac{1}{n^s}$ in a Banach space $X$ satisfies $\sigma_c(x'(D)) < +\infty$ for all $x' \in X'$, then $\sigma_c(D) < +\infty$.
\end{corollary}

It is possible to characterize those Fr\'echet $X$ that enjoy the property exhibited in Theorem \ref{cnp}. A sequence $(\gamma_n)_n \subset \C$ is rapidly decreasing if $(n^k \gamma_n)_n$ is bounded for every $k \in \N$. We refer the reader to \cite[Definition 3.1]{Vogt} for the precise definition of the topological invariant $(LB_{\infty})$ of Vogt. It is related to the (DN) type conditions of Vogt; see \cite{Vogt}. By \cite[Satz 3.2]{Vogt}, a Fr\'echet space $F$ satisfies $(LB_{\infty})$ if and only if every continuous linear operator from the Fr\'echet space $S$ of rapidly decreasing sequences into $F$ is bounded, i.e.\ maps a neighbourhood of $S$ into a bounded set of $F$. This is written $L(S,F)=LB(S,F)$ in the notation of \cite{Vogt}. The space $S$ is usually denoted as $s$, but we prefer to keep the notation $s$ for the complex numbers in this article.

\begin{theorem} \label{lbinfty}
A Fr\'echet space $X$ satisfies condition $(LB_\infty)$ of Vogt if and only if every Dirichlet series
$D = \sum_n a_n \dfrac{1}{n^s}$ in $X$ such that $\sigma_c(x'(D)) < +\infty$ for all $x' \in X'$ must also satisfy $\sigma_c(D) < +\infty$.
\end{theorem}
\begin{proof}
We set $v_k(n):=n^{-k}, n,k \in \N$. The dual $S'$ of the space $S$ of rapidly decreasing sequences coincides with $\cup_k \ell_\infty(v_k)$. Here $(\mu_n)_n \in \ell_\infty(v_k)$ if and only if $\sup_n v_k(n) |\mu_n| < +\infty$. As in \cite{BMS}, we define
$$
\overline{V}:=\{v=(v(n))_n \ | \ \forall k \ \exists \mu_k \ \forall n \ \ v(n) \leq \mu_k n^{-k} \}.
$$
A sequence $(x_n)_n$ in $X$ satisfies that $(v(n) x_n)_n$ is bounded in $X$ for every $v \in \overline{V}$ if and only if $(v(n) x_n)_n$ is weakly bounded for every $v \in \overline{V}$. By \cite[Lemma 2.1]{BMS} this is equivalent to the fact that for each $x' \in X'$ there is $k=k(x')$ such that $(n^{-k} x'(x_n))_n$ is bounded, that is $(x'(x_n))_n \in \ell_\infty(v_k)$. On account of this fact, it is easy to see that every Dirichlet series
$D$ in $X$ such that $\sigma_c(x'(D)) < +\infty$ for all $x' \in X'$ must also satisfy $\sigma_c(D) < +\infty$ if and only if for every sequence $(x_n)_n$ in $X$, such that $(v(n) x_n)_n$ is bounded in $X$, for every $v \in \overline{V}$, there is $k \in \N$ such that $(n^{-k} x_n)_n$ is bounded in $X$.

Let $(e_n)_n$ be the canonical basis of the space $S$. Mapping continuous linear maps $T:S \rightarrow X$ into the sequence $(T(e_n))_n$ we can identify the space $L(S,X)$ of all continuous linear maps from $S$ into $X$ with the space of all sequences $(x_n)_n$ in $X$ such that $(v(n) x_n)_n$ is bounded in $X$ for every $v \in \overline{V}$, as well as the space $LB(S,X)$ of all bounded linear maps from $S$ into $X$ with the space of all sequences $(x_n)_n$ in $X$ such that there is $k \in \N$ such that $(n^{-k} x_n)_n$ is bounded in $X$; see \cite[Lemma 2]{Bierstedt_Bonet}. Accordingly, every Dirichlet series
$D$ in $X$ such that $\sigma_c(x'(D)) < +\infty$ for all $x' \in X'$ must also satisfy $\sigma_c(D) < +\infty$ if and only if $L(S,X)=LB(S,X)$. The conclusion now follows from Vogt \cite[Satz 3.2]{Vogt}.
\end{proof}

Now we consider the abscissa of uniform (weak) convergence of Dirichlet series $D = \sum_n a_n \dfrac{1}{n^s}$ in $X$. To do this, we proceed as in the seminal work of Bohr \cite{Boas_Football} and define the abscissa of boundedness. Recall that $D \in H([\re s > \sigma_c(D)],X)$ by Proposition \ref{holomorphic}. We define $\sigma_b(D)$ as the infimum of all $r$ such that $D$ defines a bounded holomorphic function on $[\re s > r]$. Bohr \cite{Bohr_Uber} proved the fundamental result that $\sigma_b(D)=\sigma_u(D)$ for each scalar Dirichlet series $D$. The following result is a direct consequence of the definitions.

\begin{proposition} \label{boundedabseasy}
Let $D = \sum_n a_n \dfrac{1}{n^s}$ be a Dirichlet series in $X$.

\begin{itemize}

\item[(i)] $\sigma_b(D) = \sup_{x' \in X'} \sigma_b(x'(D))$.

\item[(ii)] Assume $-\infty < \sigma_c(D) < +\infty$. If $r > \sigma_u(D)$, then $\{D(s) \ | \ \re s \geq r \}$ is bounded in $X$. In particular, $\sigma_b(D) \leq \sigma_u(D)$.

\end{itemize}

\end{proposition}

The following result for Banach spaces is due to Defant, Garc\'{\i}a, Maestre and P\'erez-Garc\'{\i}a \cite{defantGMP}; see also \cite[Theorem 2.2]{DefantGMS}. Its proof in the Banach space case is involved and it requires a careful analysis of Bohr's arguments in the scalar case. It is a version of a fundamental result of Bohr \cite{Bohr_Uber} for sets of scalar Dirichlet series instead of a single scalar series. The proof of the version below for locally convex spaces $X$ is obtained by a reduction argument to the local Banach spaces. Given a continuous seminorm $\alpha \in cs(X)$, we denote by $X_\alpha$ the Banach space that is the completion of the normed space $X/\alpha^{-1}(0)$, endowed with the norm $\tilde{\alpha}(x + \alpha^{-1}(0))=\alpha(x), x \in X$. We write $\pi_\alpha: X \rightarrow X_\alpha$ for the canonical map.

\begin{theorem} \label{unifDGMP}
Let $D= \sum_n a_n \dfrac{1}{n^s}$ be a Dirichlet series in a sequentially complete locally convex space $X$. Then the abscissa $\sigma_u(D)$ of uniform convergence coincides with the abscissa $\sigma_b(D)$ of boundedness.
\end{theorem}
\begin{proof}
By Proposition \ref{boundedabseasy} (ii), $\sigma_b(D) \leq \sigma_u(D)$. If $\sigma_b(D) = +\infty$, there is nothing to prove. Otherwise, take $\sigma_b(D) < r(0)$ and select $\sigma_b(D) < r(1) < r(0)$. Fix $\alpha \in cs(X)$ and $\varepsilon >0$. By Proposition \ref{holomorphic}, $D(s)= \sum_n a_n \dfrac{1}{n^s}$ is holomorphic in $[\re s > r(1)]$. Moreover $\{ D(s) \ | \ \re s > r(1) \}$ is bounded in $X$. The continuity of $\pi_\alpha$ implies that the abscissa of convergence of the Dirichlet series $\pi_\alpha(D):= \sum_n \pi_\alpha(a_n) \dfrac{1}{n^s}$ in $X_\alpha$ is smaller or equal than $\sigma_c(X)$. Therefore
$\pi_\alpha(D)$ defines a holomorphic function $g_\alpha \in H([\re s > r(1)],X_\alpha)$. Moreover, $g_\alpha(s) = \pi_\alpha(D(s))$ for each $s \in \C$ with $\re s > r(1)$, hence $g_\alpha$ is holomorphic and bounded in $[\re s > r(1)]$, i.e.\ $g_\alpha \in H_\infty([\re s > r(1)],X_\alpha)$, and $\sigma_b(\pi_\alpha(D)) \leq r(1)$. We apply \cite[Theorem 2.2]{DefantGMS} to conclude that $\sum_n \pi_\alpha(a_n) \dfrac{1}{n^s}$ converges uniformly in $X_\alpha$ on $[\re s > r(0)]$. Therefore, given $\varepsilon>0$ there is $N_0 \in \N$ such that if $N>M \geq N_0$, then
$$
\tilde{\alpha}\big(\sum_{n=M}^N \pi_\alpha(a_n) \dfrac{1}{n^s}\big) = \alpha\big(\sum_{n=M}^N a_n \dfrac{1}{n^s}\big) < \varepsilon
$$
for each $s \in \C$ with $\re s > r(0)$. The proof is complete.
\end{proof}

\begin{corollary} \label{uequalsweaku}
$\sigma_u(D) = \sup_{x' \in X'} \sigma_u(x'(D))$ for every Dirichlet series $D= \sum_n a_n \dfrac{1}{n^s}$ in $X$.
\end{corollary}
\begin{proof}
By Theorem \ref{unifDGMP}, Proposition \ref{boundedabseasy} (i) and Bohr's fundamental theorem for the scalar case, we have  $\sigma_u(D)=\sigma_b(D)=\sup_{x' \in X'} \sigma_b(x'(D))=\sup_{x' \in X'} \sigma_u(x'(D))$.
\end{proof}

\section{Abscissa of absolute convergence} \label{sect_a}

In this section we compare the abscissas of absolute convergence $\sigma_a(D)$ and weak convergence $\sigma^w_a(D)$ for Dirichlet series $D=\sum_n a_n \dfrac{1}{n^s}$ in a sequentially complete locally convex space $X$. We start with the following easy example showing that $\sigma_a(D) \neq \sigma^w_a(D)$ in general, contrary to what happens for the abscissas of convergence and uniform convergence.

\begin{example} \label{excanonical}
Let $D= \sum_n \frac{e_n}{n^s}$, where $(e_n)_n$ is the canonical basis of $X=\ell_p, 1 \leq p < +\infty$ or $X=c_0$. It is easy to see that $\sigma_a(D)=1$ in all cases, but $\sigma^w_a(D)=1/p, 1 \leq p < +\infty,$ and $\sigma^w_a(D)=0$ if $X=c_0$.
\end{example}

For a sequentially complete locally convex space $X$, the \textit{gap for absolute convergence of Dirichlet series in $X$} is defined by $G_a(X):= \sup_{D} (\sigma_a(D) -  \sigma^w_a(D))$, where the supremum is taken over all the Dirichlet series $D$ with coefficients in $X$ such that $\sigma_c(D) < +\infty$. Since $\sigma_c(D) = \sigma^w_c(D) \leq \sigma^w_a(D) \leq \sigma_a(D) \leq \sigma_c(D) + 1$ by Propositions \ref{elemnetaryprop} (v) and \ref{c_equals_weak_c}, we have $0 \leq G_a(X) \leq 1$ for every space $X$. If $X$ is finite dimensional, then $G_a(X)=0$. Observe that Example \ref{excanonical} implies $G_a(\ell_p) \geq 1 - 1/p, 1 \leq p < +\infty,$ and $G_a(c_0) =1$.

A sequence $(x_n)_n$ in a sequentially complete locally convex space $X$ is called \textit{absolutely summable} if $\sum_n \alpha(x_n) < +\infty$ for each $\alpha \in cs(X)$. The sequence $(x_n)_n$ is \textit{unconditionally summable} if for every permutation $\pi$ of $\N$, the series $\sum_n x_{\pi(n)}$ converges in $X$. The sequence $(x_n)_n$ is \textit{weakly unconditionally convergent} if it is unconditionally convergent for the weak topology $\sigma(X,X')$. This is equivalent to the fact that $\sum_n |x'(x_n)|< +\infty$ for each $x' \in X'$ by Riemann's rearrangement theorem. More information about these concepts can be seen in \cite[Section 14.6]{Jarchow}, \cite{Pietsch} and, for Banach spaces, in \cite{DiJaTo} and \cite{Kadets}.

\begin{proposition} \label{gaplowerboundBanach}
For every infinite dimensional Banach space $X$, $G_a(X) \geq 1/2$.
\end{proposition}
\begin{proof}
Take $1>r>1/2$. The sequence $(1/n^r)_n$ belongs to $\ell_2$. We apply Dvoretzky-Rogers Theorem \cite[Theorem 1.2]{DiJaTo} to find an unconditionally summable sequence $(a_n)_n$ in $X$ such that $||a_n||=1/n^r$ for each $n \in \N$. For each $x' \in X'$, we have $\sum_n |x'(a_n)| < +\infty$. The Dirichlet series $D=\sum_n a_n \dfrac{1}{n^s}$ in $X$ satisfies $\sigma^w_a(D) \leq 0$, and $\sigma_a(D) = 1- r$. Therefore $\sigma_a(D) -  \sigma^w_a(D) \geq 1-r$. This implies $G_a(X) \geq 1-r$ for each $r>1/2$, and the conclusion follows.
\end{proof}

The following lemma will be useful later.

\begin{lemma}\label{c0}
If the sequentially complete locally convex space $X$ satisfies $G_a(X)<1$, then every weakly summable sequence in $X$ is unconditionally summable.
\end{lemma}
\begin{proof}
If $G_a(X)<1$, the space $X$ cannot contain a copy of the Banach space $c_0$, since otherwise $1= G_a(c_0) \leq G_a(X) \leq 1$, by Example \ref{excanonical} for the Banach space $c_0$. This is a contradiction. The conclusion now follows from a result in \cite{QR} that is an extension of a classical theorem of Bessaga and Pelczy\'nski; see e.g.\ \cite[Theorem 6.4.3]{Kadets}.
\end{proof}

Let $2 \leq p < \infty$. A Banach space $Y$ is
said to have \textit{cotype $p$} whenever there is some constant $C>0$ such that for each choice of
finitely many vectors $x_1, \ldots, x_N \in Y$ we have
\[
\Big( \sum_{n=1}^N \|x_n\|^p \Big)^{1/p} \leq C \Big(\int_0^1 \big\| \sum_{n=1}^N r_n(t)x_n
\big\|^2 dt\Big)^{1/2}\,,
\]
where $r_n$ stands for the $n$th Rademacher function on $[0,1]$. Every Banach space $Y$
has cotype $\infty$ since $\max_n \|x_n\|$ is always  dominated by the
Rademacher average of the $x_n$\,. As usual we write
\[
\Cot (Y) := \inf \{2 \leq p \leq \infty\, | \, Y\,\, \text{has cotype $p$}\}\,.
\]

An operator $T:X \rightarrow Y$ between Banach spaces $X$ and $Y$  is \textit{$(p,1)$-summing}, $1 \leq p \leq \infty$,
whenever there is a constant $c>0$
such that for each choice of finitely many $x_1, \ldots, x_n \in X$ we have that $(\sum_i
\|Tx_i\|^p)^{1/p} \leq c\, \sup_{\|x^\ast\|\leq 1} \sum_i |x^\ast(x_i)|\,$.
For every infinite dimensional Banach space $Y$ a fundamental result of
Maurey and Pisier \cite{Maurey_Pisier} (see also \cite[Theorem 14.5]{DiJaTo}, \cite{Rakov} and \cite{Talagrand}) shows that $\Cot (Y) = \inf
\{ 2 \leq p \leq \infty \,|\, \text{$\Id_Y$ is $(p,1)$-summing} \}\,.$
The identity $\Id_Y$ of a Banach space $Y$ is $(p,1)$-summing if and only if every weakly summable sequence $(x_n)$
in $Y$ satisfies $\sum_n ||x_n||^p < +\infty$, i.e.\ $Y$ has the \textit{$p$-Orlicz property}.

\begin{proposition} \label{upperboundBanach}
If the identity $\Id_X$ of an infinite dimensional Banach space $X$ is $(p,1)$-summing, $2 \leq p < +\infty$, then $G_a(X) \leq 1-(1/p)$.
\end{proposition}
\begin{proof}
Let $D= \sum_n a_n \dfrac{1}{n^s}$ be a Dirichlet series in $X$ with $\sigma_c(D) < +\infty$ and take $r \in \R$ with $\sigma^w_a(D) < r$.
Then $\sum_n \frac{|x'(a_n)|}{n^r}$ for each $x' \in X'$; i.e.\ $\sum_n \frac{a_n}{n^r}$ is weakly unconditionally summable. By assumption,
$\sum_n \frac{||a_n||^p}{n^{rp}} < +\infty$. Take $t > r + \frac{1}{p'}$ with $ \frac{1}{p} + \frac{1}{p'} = 1$. We apply H\"older's inequality to get
$$
\sum_n \frac{||a_n||}{n^{t}} \leq \big( \sum_n \frac{||a_n||^p}{n^{rp}} \big)^{1/p}  \big( \sum_n \frac{1}{n^{(t-r) p'}} \big)^{1/p'} < +\infty,
$$
because $(t-r)p' >1$. Therefore $\sigma_a(D) \leq r + \frac{1}{p'}$, hence $\sigma_a(D) - \sigma^w_a(D) \leq \frac{1}{p'}$. Since the series $D$ with
$\sigma_c(D) < +\infty$ is arbitrary, we conclude $G_a(X) \leq \frac{1}{p'} = 1 - \frac{1}{p}$.
\end{proof}

\begin{proposition} \label{lowerboundBanach}
If  $X$ is an infinite dimensional Banach space such that $G_a(X) <r <1$, then  the identity $\Id_X$ is $(t,1)$-summing for each $t>1/(1-r)$.
\end{proposition}
\begin{proof}
By Proposition \ref{gaplowerboundBanach}, $1/2 < r < 1$, hence $1/(1-r)>2$. Let $(x_n)_n$ be a weakly unconditionally summable sequence in $X$.
Since $G_a(X)<1$, we can apply Lemma \ref{c0} to conclude that $(x_n)_n$ is unconditionally summable, in particular $(x_n)_n$ converges to $0$ in $X$. Let $(||a_n||)_n$ be a decreasing rearrangement of $(||x_n||)_n$. As $(x_n)_n$ is unconditionally summable, $\sum_n |x'(a_n)| < +\infty$ for each $x' \in X'$, hence $\sigma^w_a(D) \leq 0$ for $D:= \sum_n a_n \dfrac{1}{n^s}$. By assumption $\sigma_a(D) < r$, thus $\sum_n \frac{||a_n||}{n^r} < +\infty$. The sequence $(\frac{||a_n||}{n^r})_n$ is decreasing, hence $\lim_{n \rightarrow \infty} n^{1-r}||a_n|| = 0$ by \cite[Theorem 3.3.1]{Knopp}. There is $M>0$ such that $||a_n|| \leq M/n^{1-r}$ for each $n \in \N$. If $t>1/(1-r)$, we have $\sum_n ||a_n||^t \leq M^t \sum_n 1/n^{t(1-r)} < +\infty$. This implies $\sum_n ||x_n||^t < +\infty$, since it is a rearrangement.
\end{proof}

\begin{theorem} \label{cotypeBanach}
Let $X$ be an infinite dimensional Banach space with cotype $p \geq 2$, then $G_a(X)= 1- 1/\Cot (X)$.
\end{theorem}
\begin{proof}
The inequality $G_a(X) \leq 1- 1/\Cot (X)$ is a direct consequence of Proposition \ref{upperboundBanach}. Suppose that $G_a(X) < r < 1- 1/\Cot (X)$.
By Proposition \ref{lowerboundBanach}, $\Id_X$ is $(t,1)$-summing for each $t>1/(1-r)$, and we can apply Maurey, Pisier's fundamental result \cite{Maurey_Pisier} to conclude that $1/(1-r) \leq \Cot (X)$. This implies $1 - 1/\Cot (X) \leq r$, a contradiction.
\end{proof}

\begin{remark}
Let $X$ be an infinite dimensional Banach space. As a consequence of Corollary \ref{uequalsweaku}, we have
$G_a(X) \leq \sup(\sigma_a(D) - \sigma_u(D)) =:T(X)$. Accordingly, the upper estimate in Theorem \ref{cotypeBanach} is a direct consequence of the main, deep Theorem 1 in \cite{defantGMP}. Observe that for a finite dimensional space $X$, $G_a(X)=0$, while $T(\C)=1/2$ by Bohnenblust, Hille Theorem (see \cite{DefantPeris}), thus the estimate $G_a(X) \leq T(X)$ has no consequence for Bohr's absolute convergence problem, that requires much deeper techniques \cite{DefantGMS}. However, in the infinite dimensional case it clarifies the role of weak unconditionally convergence of series in \cite[Theorem 1]{defantGMP}.
\end{remark}

We refer the reader to \cite[Chapter 11]{albiackalton}, \cite[Chapter 14]{DiJaTo} and \cite[Chapter 5]{Kadets} for finite representability and related concepts necessary in our next statement.

\begin{theorem} \label{nofinitecotype} Let $X$ be an infinite dimensional Banach space. The following conditions are equivalent.

\begin{itemize}

\item[(i)] $G_a(X)=1$.

\item[(ii)] $X$ does not have finite cotype.

\item[(iii)] $X$ contains $\ell^n_\infty$'s $\lambda$-uniformly for some (and then all) $\lambda >1$.

\item[(iv)] $\Id_X$ is not $(p,1)$-summing for any $2 \leq p < +\infty$.

\item[(v)] $\ell_\infty$ is finite representable in $X$.

\end{itemize}

\end{theorem}
\begin{proof}
Conditions (ii), (iii) and (iv) are equivalent by \cite[Theorem 14.1]{DiJaTo}, and (ii) and (v) are equivalent by \cite[Theorem 11.1.14 (ii)]{albiackalton}. Finally, the equivalence of (i) and (ii) follows from Propositions \ref{upperboundBanach} and \ref{lowerboundBanach}, proceeding by contradiction to prove both implications.
\end{proof}

\begin{remark}
If $X$ is an infinite dimensional Banach space such that $\ell_\infty$ is finite representable in $X$, then for each $(t_n)_n \in c_0$ such that $ t_n > 0, n \in \N,$ there is an unconditionally summable sequence $(x_n)_n$ in $X$ such that $||x_n||=t_n$ for each $n \in \N,$ by \cite[Theorem 5.2.1]{Kadets}. It is then possible to exhibit a Dirichlet series $D$ in $X$ such that $\sigma_a(D) - \sigma^w_a(D) =1$. Indeed, construct by induction a sequence $(s_n)_n \in c_0$ with $ s_n > 0, n \in \N,$ such that for each $k \in \N$ there is $n(k) \in \N$ such that $s_n > n^{-1/k}$ for $n \geq n(k)$. Now select an unconditionally summable sequence $(a_n)_n$ in $X$ with $||a_n||=s_n$ for each $n \in \N$. The Dirichlet series $D:=\sum_n \frac{a_n}{n^s}$ satisfies $\sigma^w_a(D) \leq 0$ and $\sigma_a(D) \geq 1$.
\end{remark}

\begin{theorem} \label{nuclear}
(a) If $X$ is a nuclear sequentially complete locally convex space, then $G_a(X)=0$.

\vspace{.1cm}
\noindent
(b) The following conditions are equivalent for a Fr\'echet space $X$:

\begin{itemize}

\item[(i)] $X$ is nuclear.

\item[(ii)] $G_a(X)=0$.

\item[(iii)] $G_a(X) < 1/2$.

\end{itemize}
\end{theorem}
\begin{proof}
The proof of (a) is a consequence of the fact that weakly unconditionally summable sequences in a nuclear locally convex space are absolutely summable; see e.g.\ \cite[Proposition 4.2.2]{Pietsch}.
Now only (iii) implies (i) needs a proof in statement (b). Assume that $X$ is a Fr\'echet space such that $G_a(X) < 1/2$. Fix $G_a(X) < r < 1/2$ and select $q \in ]1/(1-r),2[$. Let $(x_n)_n$ be a weakly unconditionally summable sequence in $X$. By Proposition \ref{c0}, $(x_n)_n$ is unconditionally summable, hence it converges to $0$ in $X$. Fix $\alpha \in cs(X)$. Since $(\alpha(x_n))_n$ tends to $0$, we can reorder it in a decreasing way. Denote by $(\alpha(y_n))_n$ the reordered sequence, that depends on $\alpha$. Since $\sum_n y_n$ is a rearrangement of $\sum_n x_n$ (which is weakly unconditionally summable), we have $\sum_n |x'(y_n)| < +\infty$ for each $x' \in X'$. Therefore $\sigma^w_a(\sum_n \frac{y_n}{n^s}) \leq 0$. By assumption $\sigma_a(\sum_n \frac{y_n}{n^s}) < r$, hence $\sum_n \frac{\alpha(y_n)}{n^r} < +\infty$. As $(\alpha(y_n)/n^r)_n$ is decreasing, we can apply \cite[Theorem 3.3.1]{Knopp} to conclude $\lim_{n \rightarrow \infty} n^{1-r} \alpha(y_n) = 0$. There is $M>0$ such that $\alpha(y_n) \leq M n^{1-r}$ for each $n \in \N$, hence $\sum_n \alpha(y_n)^q \leq M^q \sum_n \frac{1}{n^{q(1-r)}} < +\infty$, since $q(1-r) >1$. Consequently $\sum_n \alpha(x_n)^q < +\infty$, because it is a rearrangement. Since $\alpha \in cs(X)$ is arbitrary, we have shown that there is $q \in ]1,2[$ such that every weakly unconditionally summable sequence $(x_n)_n$ in $X$ is $q$-absolutely summable (i.e.\ $\sum_n \alpha(x_n)^q < +\infty$ for each $\alpha \in cs(X)$).

Now the closed graph theorem for Fr\'echet spaces and standard arguments (see \cite{Floret_Wloka},  \cite[Theorem 21.2.1]{Jarchow}, \cite[Proposition 28.4]{Meise_Vogt} and \cite[Section 4.2]{Pietsch}) permit us to conclude that for each $\alpha \in cs(X)$ there is $\beta \in cs(X)$ such that the canonical linking map $X_\beta \rightarrow X_\alpha$ is $(q,1)$-summing. Since $1<q<2$, and the composite of sufficiently many $(q,1)$-summing maps ($1 < q < 2$) produces a nuclear map by \cite[Corollary 5.7]{KonigRTJ}, we conclude that $X$ is nuclear.
\end{proof}

\begin{example}

\begin{itemize}

\item[(i)] As a consequence of Theorems \ref{cotypeBanach}, $G_a(\ell_p) = 1/2$ if $1 \leq p \leq 2$ and $G_a(\ell_p)=1 - 1/p$ if $2 \leq p \leq \infty$. In fact, in this case the lower estimates are a direct consequence of Proposition \ref{gaplowerboundBanach} and Example \ref{excanonical}. Compare with \cite[Corollary 3]{defantGMP}.

\item[(ii)] If a complete locally convex space $X$ is a projective limit of infinite dimensional Banach spaces $X_\gamma, \gamma \in \Gamma,$ such such that each $X_\gamma$ is of cotype $2 \leq p < +\infty$, then $G_a(X) \leq 1-1/p$. This follows from the definitions and Proposition \ref{upperboundBanach}.

\item[(iii)] Let $1 \leq p < \infty$ and let $X=\ell_{p+}$ be the Fr\'echet space defined as the intersection of all $\ell_q$ space with $q>p$. Then $G_a(\ell_{p+})=G_a(\ell_p)$, although  $\ell_{p+}$ is a non-Montel Fr\'echet space that contains no Banach space \cite{diaz}.

\item[(iv)] Every non-Montel K\"othe echelon space $\lambda_p(A)$ of order $1\leq p < +\infty$ contains a sectional subspace isomorphic to $\ell_p$; see \cite[Theorem 27.9]{Meise_Vogt}. Therefore $G_a(\lambda_p(A))=G_a(\ell_p)$ for every non-Montel K\"othe echelon space $\lambda_p(A)$.

\item[(v)] For every $t \in [1/2,1]$ there are Banach spaces $X_t$ and non-normable Fr\'echet spaces $Y_t$ such that $G_a(X_t)=G_a(Y_t)=t$.

\end{itemize}

\end{example}

\textbf{Acknowledgement.} The research of this paper was partially
supported by the projects MTM2013-43540-P and GVA Prometeo II/2013/013 (Spain).




\noindent \textbf{Author's address:}%
\vspace{\baselineskip}%

Instituto Universitario de Matem\'{a}tica Pura y Aplicada IUMPA,
Universitat Polit\`{e}cnica de Val\`{e}ncia,  E-46071 Valencia, SPAIN

email:jbonet@mat.upv.es

\end{document}